\documentclass[10pt]{amsart}%
\usepackage{amsfonts}
\usepackage{amssymb}
\usepackage{amsmath}
\usepackage{url}
\usepackage{amsmath}
\usepackage{graphicx}%
\providecommand{\U}[1]{\protect\rule{.1in}{.1in}}
\usepackage{amsthm}
\usepackage{enumerate}
\usepackage{url}
\theoremstyle{plain}
\newtheorem{example}{Example}
\newtheorem{definition}{Definition}
\newtheorem{theorem}{Theorem}
\newtheorem{remark}[theorem]{Remark}

\newtheorem{corollary}[theorem]{Corollary}
\newtheorem{proposition}[theorem]{Theorem}
\newtheorem{lemma}[theorem]{Lemma}
\newtheorem*{thma}{Main Theorem}

\def\d{{\rm dom}\hphantom{.}}

\def\2{{\bf 2}}
\def\P2{{\rm Par}(\2)}
\def\PmA{{\rm Par}^{(m)}(\k)}
\def\PnA{{\rm Par}^{(n)}(\k)}
\def\Pn2{{\rm Par}^{(n)}(2)}
\def\PA{{\rm Par}(\k)}
\def\pp{{\rm pPol}\,}

\def\st{~|~}

\def \imp{\Longrightarrow}

\def\va{{\vec a}}

\def\r{\rho}

\def\0{\vec{0}}
\def\1{\vec{1}}
\def\k{{\bf k}}

\def\J2{{\rm J}(\2)}

\def\s{\sigma}
\def\pne{{\pp(\ne)}}
\def\ple{{\pp(\le)}}

\def\PN{{\mathcal P}({\bf E}_{\ge 4})}

\begin{document}
\title[Partial clones containing Boolean monotone self-dual p. funct.s]{Partial clones containing all Boolean monotone self-dual partial functions}
\author[M. Couceiro]{Miguel Couceiro}
\address[Miguel Couceiro]{LORIA (CNRS - Inria Nancy G.E. - Universit\'e de Lorraine), Equipe Orpailleur,
Batiment B, Campus Scientifique,   B.P. 239,  
 F-54506 Vandoeuvre-l\`es-Nancy}
 \email{miguel.couceiro@inria.fr}
\author[L. Haddad]{Lucien Haddad}
\address[Lucien Haddad]{D\'epartment de Math\'ematiques et d’Informatique, Coll\`ege militaire royal du Canada,
boite postale 17000, STN Forces, Kingston ON K7K 7B4 Canada.}
\email{haddad-l@rmc.ca }
\author[I. G. Rosenberg]{Ivo G. Rosenberg}
\address[Ivo G. Rosenberg]{D\'epartment de Math\'ematiques et Statistique, Universit\'e de Montr\'eal, Montr\'eal,
Qu\'ebec, H3C 3J7 Canada}
\email{rosenb@dms.umontreal.ca}
\thanks{This work was done while the first named author was a Research Assistant at the Mathematics Research Unit at the University of Luxembourg, Luxembourg, 
and an Associate Professor at LAMSADE at Universit\'e Paris-Dauphine, France.\\
The second named author wishes to acknowledge the financial support by ARP grants.\\
Les trois auteurs veulent rendre hommage \`a Maurice Pouzet. Maurice est un math\'e\-maticien brillant et
respect\'e, ses contributions scientifiques sont immenses. Tous ceux qui le connaissent appr\'ecient non
seulement sa profonde culture scientifique mais aussi ses grandes qualit\'es humaines. Maurice est plein
d’\'geards pour son entourage, il est g\'en\'ereux et tr\`es d ́\'evou\'e envers ses amis. Il prend plaisir \`a encourager, 
aider et rendre service aux autres. Nous avons la chance et le privil\`ege d’avoir Maurice comme
collaborateur et ami. Merci pour tout Maurice!}

\begin{abstract} The study of partial clones on $\2:=\{0,1\}$  was initiated by R. V. Freivald. In his fundamental paper published in 1966,
Freivald showed, among other things, that the set of all monotone partial functions and the set of all self-dual partial 
functions are both maximal partial clones on $\2$.

Several papers dealing with intersections of maximal partial clones on $\2$  have appeared after Freivald work. It is known that there are infinitely many  partial clones that contain the set of all monotone self-dual  partial functions on $\2$, and the problem of describing them all was posed by some authors.

In this  paper we show that the set of partial clones that contain all  monotone self-dual partial functions is of continuum cardinality on $\2$.
\end{abstract}
\maketitle

\section{Preliminaries}
   Let $A$ be a finite non-singleton set. Without loss of generality we assume that $A=\k:=\{0,\dots,k-1\}$.   For a positive integer
$n$, an {\ $n$-ary partial function\/} on $\k$ is a map $f: \d(f) \to
\k$ where $\d(f)$ is a subset of $\k^n$ called the {\it domain} of $f$.
Let $\PnA$ denote the set of all $n$-ary
partial functions on $\k$ and let 
$$\PA := \bigcup\limits_{n\ge 1}\PnA.$$

    For
$n, m \ge 1$, $f \in \PnA$ and $g_1, \dots, g_n \in \PmA$,  the
{\it composition} of $f$ and $g_1, \dots, g_n$, denoted by
$f[g_1, \dots, g_n] \in \PmA$, is defined by
\begin{eqnarray*}
\d(f[g_1, \dots, g_n])&:=& \{\va \in \k^m \st \va \in \bigcap\limits^m_{i=1} \d (g_i)\\
&&{\rm and}\;  (g_1(\va), \dots, g_m(\va))\in \d(f)\} 
\end{eqnarray*}

and 
\begin{eqnarray*}
 f[g_1, \dots, g_n](\va):=f(g_1(\va),\dots,
g_n(\va)),
\end{eqnarray*}
for all $\va \in \d( f[g_1, \dots, g_n]).$

   For every positive integer $n$ and each
$1 \le i \le n$, let
$e_{i}^n$ denote the {\it $n$-ary i-th projection function} defined by
  $$e^n_{i} (a_1, \dots, a_n) = a_i$$
  for all $(a_1, \dots, a_n) \in \k^n$.
 Furthermore, let
$$
J_\k:=\left\{e^n_{i} : 1 \le i\le n \right\}
$$
be the set of all (total) projections.
\vskip 0.1in
  
\begin{definition} {\rm A {\it partial clone} on $\k$ is a
composition closed subset of $\PA$ containing $J_\k$. }
\end{definition}

\vskip 0.1in

\begin{remark}
 There are two other equivalent  definitions for partial clones. One definition uses
Mal'tsev's  formalism and the other uses the concept of one point extension. These definitions can be found in chapter 20 of \cite{booklau}.
\end{remark}

The partial clones on $\k$, ordered by inclusion, form a lattice ${\mathcal L}_{P_\k}$
in which the infinimum is the set-theoretical intersection. That means that the intersection of an arbitrary family of partial clones on $\k$ is also a partial clone on $\k$.
A \emph{maximal partial clone} on $\k$
is a coatom of the lattice ${\mathcal L}_{P_\k}$. Therefore  a partial clone $M$ is  maximal if there is no partial clone
$C$ over $\k$ such that $M \subset C \subset \PA$.

\begin{example} The set of partial functions
$$\Omega_k:=\bigcup\limits_{n\ge 1} \{f \in \PnA \st \d(f) \ne \emptyset \imp \d(f)=\k^n\}$$
 is a maximal partial clone on $\k$.
\end{example}

\begin{definition} \label{preserves}   For $h \ge 1$, let $\r$ be an $h$-ary relation on $\k$ and $f$ be
an $n$-ary partial function on $\k$. We say that $f~preserves~
\r$ if for every $h \times n$ matrix $M=[M_{ij}]$ whose columns
$M_{*j} \in \r,~(j=1,\ldots n)$ and whose rows $M_{i*} \in \d
(f)$ $(i=1,\ldots,h)$, the $h$-tuple
$(f(M_{1*}),\ldots,f(M_{h*})) \in \r$.  Define
$$\pp\r:=\{f\in \PA \st f~{\rm preserves}~\r\}.$$
It is well known that $\pp \r$ is a partial clone called the {\it partial clone determined by
the relation} $\r$.
\end{definition}

Notice that if there is no  $h \times n$ matrix $M=[M_{ij}]$ whose columns
$M_{*j} \in \r$ and whose rows $M_{i*} \in \d
(f)$, then $f \in \pp \r$.

\begin{example} Let $\2 :=\{0,1\}$ and  let $\{(0,0),(0,1),(1,1)\}$ be the natural order on $\2$.
Consider the binary relation $\{(0,1),(1,0)\}$ on $\2$.
Then 
$$\pp \{(0,0),(0,1),(1,1)\} $$ is the set of all {\it monotone} partial functions
and 
$$\pp \{(0,1),(1,0)\}$$ 
is the set of all {\it self-dual} partial functions
on $\2$. 

For simplicity  we will write  $\ple$ and $\pne$ for  
$$\pp(\{(0,0),(0,1),(1,1)\}) \quad \text{and}\quad \pp(\{(0,1), (1,0)\}),$$ 
respectively. It is not difficult to see that
\begin{eqnarray*}
 \ple  &:=&\{f \in \P2 \st ~[\mathbf{a},\mathbf{b} \in \d(f),
\mathbf{a} \le \mathbf{b}] 
\imp f(\mathbf{a})\le f(\mathbf{b})\}, ~ \text{and}\\
\pne&:=&\{f \in \P2 \st ~[ \mathbf{a}, \mathbf{a}+1 \in \d(f)] \imp f(\mathbf{a}+1)=f(\mathbf{a})+1\}
\end{eqnarray*}
where the above sums are taken mod 2.
\end{example}

As mentioned earlier, Freivald showed that there are exactly eight maximal partial clones on $\2$. 
The following two relations are needed to state Freivald's result.
 Set 
 \begin{eqnarray*}
  R_1&:=&\{(x,x,y,y) \st x, y \in \2\} \cup \{(x,y,y,x) \st x, y
\in \2\}\quad \text{and}\\
R_2&:=& R_1 \cup \{(x,y,x,y) \st x, y \in \2\}.
 \end{eqnarray*}

\begin{theorem}[\cite{freivald}] \label{freivald} 
There are exactly 8 maximal partial clones on $\2$, namely, $\pp \{0\}$,  $\pp\{1\}$,
  $\pp\{(0,1)\})$, 
$\pp (\le)$, 
 $\pp(\ne)$, 
 $\pp (R_1)$,
 $\pp( R_2)$, and 
  $\Omega_2$. 
\end{theorem}

Notice that the total functions in $\pp R_2$ (i.e., the functions with full domain) form the maximal clone of all (total) 
linear functions over $\2$ (see, e.g., chapter 3 of  \cite{booklau}).

An interesting and somehow difficult problem in clone theory is to study  intersections
of maximal partial clones.  It is shown in \cite{alekzeev} that
the set of all partial clones on $\2$ that contain the maximal
clone consisting of all {\it total} linear functions on $\2$ is
of continuum cardinality (for details see \cite{alekzeev, haddadlau} and Theorem 20.7.13 of \cite{booklau}).
 A consequence of this is that the interval of partial clones $[\pp(R_2) \cap \Omega_2, \P2]$ is of continuum cardinality on  $\2$.
 
A similar result, (but slightly easier to prove) is established in
\cite{haddad} where it is shown that the interval of partial clones $[\pp(R_1) \cap \Omega_2, \P2]$ is also of continuum cardinality.
Notice that the three maximal partial clones $\pp R_1,~\pp R_2$ and  $\Omega_2$ contain all unary functions (i.e., maps) on $\2$. 
Such partial clones are called {\it S\l upecki type} partial clones in \cite{haddadlau,romov2}. 
These are the only three maximal partial clones of S\l upecki type on $\2$.

For a complete study  of  the pairwise intersections
of all maximal partial clones of S{\l}upecki type on a finite non-singleton set $\k$, see \cite{haddadlau}.

 The papers
\cite{italian,infinite,lau-scho,strauch1,strauch2}  focus on the case
$k=2$ where various interesting, and sometimes hard to obtain, results are established. 

For instance, the intervals
$$ [\pp\{0\} \cap \pp\{1\}\cap \pp\{(0,1)\}\cap \pp(\le), ~ \P2]\quad \text{and}$$
$$ [\pp\{0\} \cap \pp\{1\}\cap \pp\{(0,1)\}\cap \pp(\ne), ~ \P2]$$ 
 are shown to be finite and are completely described in \cite{italian}. 
 Some of the results in \cite{italian} are included in \cite{strauch1,strauch2} where partial clones on $\2$ are handled via the one point extension approach 
 (see Section 20.2 in \cite{booklau}).

 In view of results from \cite{alekzeev,haddad,italian,strauch1,strauch2}, it was thought that if $2 \le i \le 5$ and $M_1, \dots,M_i$ 
 are non-S\l upecki maximal partial clones on $\2$, then the interval
 $$[ M_1 \cap \dots \cap M_i, \P2]$$
 is either finite or countably infinite. 
 
 Now it was shown in \cite{infinite} that the interval of partial clones $[\ple \cap \pne, \P2]$ is infinite. 
 This result is mentioned  in Theorem 20.8 of \cite{booklau} (with an independent proof given in \cite{lau-scho})   and in chapter 8 of the PhD thesis \cite{scholzel}.  
 However, it remained an open problem to determine whether $[\ple  \cap \pne,  \P2]$ is countably or uncountably infinite.

In this paper we settle this question by proving that  the interval of partial clones  $$[\ple  \cap \pne,  \P2]$$ is of continuum cardinality on $\2$.

\section{The construction}
\vskip0.3cm

For $n \ge 5$ and $n  >  k > 1$ we denote by  $\s ^n_k \subseteq \2^{2n}$   the $(2n)$-ary relation defined by
\begin{eqnarray*}
\s^n_k:=\{(x_1,\dots,x_n,y_1, \dots,y_n) \in \2^{2n} \st ~\forall ~i=1,\dots,n,~x_i  \ne y_i, \quad
\text{and }\\
\forall ~i=1,\dots,n,~ y_{i+1} \le x_{i}~ \text{ and }~ y_{i+2} \le x_{i} \dots~\text{ and }~ y_{i+k} \le x_{i}\},
\end{eqnarray*}
 where the subscripts $i+j$ in the above definition are taken modulo $n$. It is not difficult to see that
\begin{eqnarray*}
\s^n_k:=\{(x_1,\dots,x_n,y_{1}, \dots,y_{n}) \in \2^{2n} \st  ~\forall ~i=1,\dots,n,~x_i  \ne y_{i}, \quad
\text{and }\\
\forall ~i=1,\dots,n,~ x_{i} =0 \imp [ x_{i+1}=x_{i+2}= \dots = x_{i+k}= 1]\}.
\end{eqnarray*}
By the {\it Definability Lemma} established by B. Romov  in \cite{romov}  (see also Lemma 20.3.4 in \cite{booklau} and \cite{haddadlau,italian,infinite} for details), we have that 
$$\ple  \cap \pne  \subseteq \pp (\s^n_k)$$ for all $n \ge 5$ and all $k \ge 1$.

For $n \ge 5$ and $n  >  k \ge 1$, we denote by  $\r ^n_k \subseteq \2^{4n}$   the $(4n)$-ary relation defined by
\begin{eqnarray*}
\r^n_k:=\{(x_1,\dots,x_n,x_{n+1}, \dots,x_{2n},y_1, \dots,y_n,y_{n+1}, \dots,y_{2n}) \in \2^{4n} \st\\
 (x_1,\dots,x_n,y_1,\dots,y_n)\in \s^n_1, ~\text{ and }~(x_{n+1},\dots,x_{2n},y_{n+1},\dots,y_{2n})\in \s^n_k\}.
\end{eqnarray*}
Again by the Definability Lemma,  we have that 
$$\pp (\s^n_1) \cap \pp (\s^n_k) \subseteq \pp(\r^n_k),$$ and thus  $\ple  \cap \pne  \subseteq \pp (\r^n_k)$ for all $n \ge 5$ and all $k \ge 1$.

Our goal is to construct an infinite set of odd integers $X$ and an infinite family of partial functions $\{g_t,~t \in X\}$ so that for every 
$t,  t' \in X$, we have $g_t\in \pp \r^{{n(t')}}_{t'}$ if and only if $t \ne t'$.

\begin{remark}\label{smallerrelations}
 Since every tuple in $\s^n_k$ (resp. $\r^n_k$) is completely determined by its first $n$ entries (resp. 2$n$ entries), we will omit the second half of such tuples.  
We therefore denote by $S^n_k$ and  $R^n_k$  the relations obtained from $\s^n_k$ and $\r^n_k$, respectively, by deleting the second half of every tuple in $\s^n_k$ and $\r^n_k$, i.e.,
$$S^n_k:= \{(x_1, \dots, x_{n}) \in \2^n \st (x_1, \dots, x_{n},1+x_1, \dots,1+x_{n})\in \s^{n}_k\}$$
and
$$R^n_k:= \{(x_1, \dots, x_{2n}) \in \2^{2n} \st (x_1, \dots, x_{2n},1+x_1, \dots,1+x_{2n})\in \r^{n}_k\}$$
where the above sums are taken mod 2.
\end{remark}

Note that $S^n_k$ is the $n$-ary relation on $\2$  whose members are tuples  in which any two $0$'s are separated by at least $k$ symbols $1$
(in particular, if the first position is $0$, then the last $k$ positions must be $1$). Furthermore,
 $R^n_k $ is the cartesian product $S^n_1 \times S^n_k$.

As mentioned earlier we will use the relations $S^n_1,  S^n_k$ and $R^n_k$ with the understanding that we are omitting the second parts of the relations 
$\s^n_1, \s^n_k$ and $\r^n_k$ in order to simplify the notation.

\medskip
{\bf Notations.}~  In the sequel  $k\geq 4$ stands for an {\bf even} integer. Set
 $n(k):=k(k+1)+1$. We  will write $\r^{n(k)}$ for $\r^{n(k)}_k$ and  $R^{n(k)}_{}$ for $R^{n(k)}_k$.
Let $M^k_{\uparrow}$ be the $n(k)\times n(k)$ matrix with columns in
$S^{n(k)}_1$, the first being $c_1=[0110101\dots0101]^T$ and
the remaining columns are obtained by applying cyclic shifts to $c_1$, i.e.,
\begin{eqnarray*}
 c_2&=&[10110101\cdots010]^T,\\
 c_3&=&[010110101\cdots01]^T,\\
 &\cdots&\\
 c_{n(k)}&=&[110101\cdots010]^T.
\end{eqnarray*}

\begin{remark}\label{rem:1} Let $r_i$ and $r_j$ be two rows of $M^k_{\uparrow}$.
If $|i-j|\geq 2 ({\rm mod}\,n(k))$, then $r_i$ and $r_j$ have a $0$ in the same position.
\end{remark}

\begin{lemma}\label{lem:1}
If $k'<k$, then there is no  $n(k')\times n(k)$ matrix $N$ whose columns are in $S^{n(k')}_1$
and whose rows are rows of $M^k_{\uparrow}$.
\end{lemma}

\begin{proof} Suppose that $k'<k$ and that $N$ is an $n(k')\times n(k)$ matrix whose columns are in $S^{n(k')}_1$.
Suppose, by way of contradiction, that the rows of $N$ are rows of $M^k_{\uparrow}$. By Remark~\ref{rem:1}, the only possible adjacent rows
of a row $r$ in $N$ are exactly the predecessor and successor rows of $r$ in $M^k_{\uparrow}$.
But then $n(k')$ would be even, thus yielding the desired contradiction.
\end{proof}

Let $M^k_{\downarrow}$ be the $n(k)\times n(k)$ matrix with columns in
$S^{n(k)}_k$, and such that the first is
$c'_1=[0\underbrace{1\cdots 1}_{k+1}0\underbrace{1\cdots 1}_{k}\cdots0\underbrace{1\cdots 1}_{k}]^T$ and
the remaining columns are obtained by applying cyclic shifts to $c_1'$ as before.

\begin{remark}\label{rem:2} Since $k\geq 4$ is even, if $r_i$ is a row of $M^k_{\uparrow}$, and $r'_j$ is a row of $M^k_{\downarrow}$,
then $r_i$ and $r'_j$ have a $0$ in the same position.
\end{remark}

\begin{lemma}\label{lem:2}
If $k'>k$, then there is no  $n(k')\times n(k)$ matrix $N$ whose columns are in $S^{n(k')}_k$
and whose rows are rows of $M^k_{\downarrow}$.
\end{lemma}

\begin{proof} Suppose that $k'>k$ and that $N$ is an $n(k')\times n(k)$ matrix whose columns are in $S^{n(k')}_k$.
Assume, by way of contradiction,  that the rows of $N$ are rows of $M^k_{\downarrow}$.
Since each row of  $M^k_{\downarrow}$ has exactly $k$ $0$'s, we have that $N$ has $k\times n(k')$ $0$'s.
Hence the matrix $N$ has a column with at least $\frac{k\times n(k')}{n(k)}$ symbols 0. 
It is easy to verify that since $k' > k \ge 4$, we have that $\frac{k\times n(k')}{n(k)}>k'$. 
But this yields the desired contradiction, since all columns of $N$ are members of  $S^{n(k')}_k$, and each has at most $k'$ $0$'s.
\end{proof}

Define $M_k$ as the $2n(k)\times n(k)$ matrix given by
\[
M_k = \left(
\begin{array}{c}
M^k_{\uparrow}\\
M^k_{\downarrow}\\
 \end{array}
\right).
\]
Notice that each column of $M_k$ is a tuple of $R^{n(k)}$.

\begin{lemma}\label{lem:3}
Let $N$ be a $2n(k')\times n(k)$ matrix whose columns are in $R^{n(k')}$
and whose rows are rows of $M^k$. Then, either all rows of $N$ are rows of  $M^k_{\downarrow}$, or
the first $n(k')$ are rows of  $M^k_{\uparrow}$ and the remaining $n(k')$ are rows of  $M^k_{\downarrow}$.
\end{lemma}

\begin{proof}
By Remark~\ref{rem:1} and the fact that $R^{n(k')}:=S^{k'}_1\times S^{k'}_k$, there cannot be more than 2 rows of $M^k_{\uparrow}$ among the last $n(k')$.
In fact, by Remark~\ref{rem:2} there can only be rows from  $M^k_{\downarrow}$ among the last $n(k')$ rows of $N$.
Furthermore, from Remark~\ref{rem:2} and the fact that $R^{n(k')}:=S^{k'}_1\times S^{k'}_k$, it follows that either all of the first  $n(k')$ 
rows of $N$ are rows of $M^k_{\uparrow}$ or all of the first  $n(k')$ rows of $N$ are rows of $M^k_{\downarrow}$.
\end{proof}

Let $f_k$ be the $n(k)$-ary partial function whose domain is the set of rows of $M_k$, and such that
$f_k$ is constant $0$ on the rows of $M^k_{\uparrow}$ and constant $1$ on the rows of $M^k_{\downarrow}$.

\begin{proposition} \label{main}
Let $k,k'\geq 4$ be even integers. Then $f_k  \in \pp R^{n(k')}$ if and only if $k\neq k'$.
\end{proposition}

\begin{proof}
Since $[0\cdots01\cdots1]^T$ does not belong to $R^{n(k)}$, we see that $f_k \not\in \pp R^{n(k)}$.

So suppose that $k\neq k'$. If $k<k'$, then it follows from Definition \ref{preserves}  and  Lemmas~\ref{lem:2} and \ref{lem:3} that $f_k \in \pp R^{n(k')}$.

Suppose now that $k>k'$. If $N$ is an $2n(k')\times n(k)$ matrix whose columns are in $R^{n(k')}$
and whose rows are rows of $M^k$ (otherwise we are done for the domain of $f_k$ is exactly the set of rows of $M^k$),
then by  Lemmas~\ref{lem:1} and \ref{lem:3} it follows that all
 rows of $N$ are rows of  $M^k_{\downarrow}$. Since $f_k$ is constant 1 on the rows of $M^k_{\downarrow}$,
 and since the constant 1 $2n(k')$ tuple belongs to $R^{n(k')}$, we conclude that $f_k \in \pp R^{n(k')}$.
\end{proof}

Let $\overline{M_k}$ be the $2n(k) \times n(k)$ matrix obtained by replacing every row of the matrix $M_k$ by its dual tuple
(obtained by interchanging $1$'s and $0$'s) and define  $L_k$ as the $4n(k)\times n(k)$ matrix given by
\[
L_k = \left(
\begin{array}{c}
M_k\\
\overline{M_k}\\
 \end{array}
\right).
\]

Moreover, let $g_k$ be the $n(k)$-ary partial function whose domain is the set of rows of $L_k$, and such that
$g_k({\vec u}) =f_k({\vec u})$  if ${\vec u}$ is a row of $M_k$ and $g_k({\vec u}) =1+f_k({\vec u})$ (mod 2) if ${\vec u}$ is a row of $\overline{M_k}$.  
Then, Theorem \ref{main} can be restated as follows:

\begin{thma} Let $k,k'\geq 4$ be even integers. Then $g_k  \in \pp \r^{n(k')}$ if and only if $k\neq k'$.
\end{thma}

Let ${\bf E}_{\ge 4}:=\{4,6,8,\dots\}$ be the set of all even integers greater or equal to 4 and denote by $\PN$ the power set of ${\bf E}_{\ge 4}$. 
Since  $$\ple  \cap \pne  \subseteq \pp (\r^n_k)$$
for every $n \ge 5$ and every $n > k \ge 1$, we have
$$ \ple  \cap \pne  \subseteq \bigcap_{t \in {\bf E}_{\ge 4}\setminus X}\pp \r^{n(t)}$$ 
for every subset $X$ of ${\bf E}_{\ge 4}$.

So let  $X \subset {\bf E}_{\ge 4}$ and fix $k \in X$. Then $g_k  \in \pp \r^{n(t)}$ for all $t \in {\bf E}_{\ge 4}\setminus X$, i.e.,  
$$ g_k \in  \bigcap_{t \in {\bf E}_{\ge 4}}\setminus X.$$ On the other hand, if  $k \in {\bf E}_{\ge 4}\setminus X$, then we have  
$$ g_k \not\in    \bigcap_{t \in {\bf E}_{\ge 4}\setminus X}\pp \r^{n(t)}\quad \text{and}\quad g_k \not\in    \bigcap_{t \not\in X}\pp \r^{n(t)},$$ 
since $g_k  \not\in \pp \r^{n(k)}$. Therefore the map
$$\chi := \PN \to [\ple  \cap \pne,  \P2]$$
  defined by 
  $$\chi(X):=  \displaystyle  \bigcap_{t \in {\bf E}_{\ge 4}\setminus X}\pp \r^{n(t)}$$
  is one-to-one and we have shown the following result which answers our question on cardinality of the interval $[\ple  \cap \pne,  \P2]$.

  \begin{corollary}
  The interval of partial clones  $[\ple  \cap \pne,  \P2]$ is of continuum cardinality on $\2$.
  \end{corollary}

\end{document}